\documentclass{article}
\usepackage{amsmath, amsthm, amssymb}
\usepackage[margin=3.5cm]{geometry}
\usepackage{hyperref}
\usepackage{verbatim}
\usepackage{url}
\usepackage{yfonts}
\usepackage{youngtab}
\usepackage{shuffle}

\usepackage[latin1]{inputenc}
\usepackage{subfigure}

\newtheorem{thm}{Theorem}
\newtheorem{prop}[thm]{Proposition}
\newtheorem{lem}[thm]{Lemma}
\newtheorem{cor}[thm]{Corollary}

\theoremstyle{remark}
\newtheorem{rem}[thm]{Remark}

\newtheorem{exm}[thm]{Example}

\newtheorem{fact}[thm]{Fact}

\newcommand{\colvec}[2]{ \begin{bmatrix} #1 \\ #2 \end{bmatrix} }

\newcommand{\id}{\mathrm{id}}

\newcommand\numberthis{\addtocounter{equation}{1}\tag{\theequation}}

\title{Generalizations of an Expansion Formula for Top to Random Shuffles}
\author{Roger Tian}
\date{\today}

\makeatletter
\newcommand\ackname{Acknowledgements}
\if@titlepage
  \newenvironment{acknowledgements}{%
      \titlepage
      \null\vfil
      \@beginparpenalty\@lowpenalty
      \begin{center}%
        \bfseries \ackname
        \@endparpenalty\@M
      \end{center}}%
     {\par\vfil\null\endtitlepage}
\else
  
\fi
\makeatother
\begin{document}
\maketitle

\begin{abstract}
In the top to random shuffle, the first $a$ cards are removed from a deck of $n$
cards $12 \cdots n$ and then inserted back into the deck. This action can be
studied by treating the top to random shuffle as an element $B_a$, which we define formally in Section \ref{prelim}, of the algebra
$\mathbb{Q}[S_n]$. For $a = 1$, Adriano Garsia in ``On the Powers of Top to Random Shuffling'' (2002) derived an
expansion formula for $B_1^k$ for $k \leq n$, though his proof for the formula was non-bijective. We prove, bijectively, an expansion formula for the arbitrary finite product $B_{a_1}B_{a_2} \cdots B_{a_k}$ where $a_1, \ldots, a_k$ are positive integers, from which an improved version of Garsia's aforementioned formula follows. We show some applications of this formula for $B_{a_1}B_{a_2} \cdots B_{a_k}$, which include enumeration and calculating probabilities. Then for an arbitrary group $G$ we define the group of $G$-permutations $S_n^G := G \wr S_n$ and further generalize the aforementioned expansion formula to the algebra $\mathbb{Q}[S_n^G]$ for the case of finite $G$, and we show how other similar expansion formulae in $\mathbb{Q}[S_n]$ can be generalized to $\mathbb{Q}[S_n^G]$.
\end{abstract}

\tableofcontents

\section{Introduction}
\label{intro}
Shuffling is a much studied topic in probability theory and combinatorics. One mode of shuffling is the \textit{top to random shuffle}, where we remove the first $a$ cards from a deck of $n$ cards $12 \cdots n$ and then insert them back into the deck. 

The case $a=1$ is known as the Tsetlin library, which is usually studied via the characterization of randomly removing a book from a row of $n$ books and placing it at the right end of the row. The Tsetlin library has been much studied as a Markov chain, where the probability of removing each book is assumed to be known. We first briefly mention some work in the literature on this subject and its generalizations.

Hendricks \cite{hend} found the stationary distribution of the Tsetlin library Markov chain. Fill \cite{fill} derived a formula for the distribution of the Tsetlin library after any number of steps, as well as the eigenvalues of this Markov chain. \cite{fill} also contains a wealth of references on this subject. Diaconis, Fill, and Pitman \cite{diaconistop} gave the the distribution of the top to random $m$-shuffle (where the first $m$ cards are moved) after any number of steps. Bidigare, Hanlon, and Rockmore \cite{bidigare} generalized the Tsetlin library to the setting of hyperplane arrangements, and calculated the eigenvalues for this generalized Markov chain. Brown \cite{brown} treated random walks on a class of semigroups called ``left regular bands'', which includes the hyperplane chamber walks of Bidigare, Hanlon, and Rockmore. Uyemura Reyes \cite{reyes} studied the random to random shuffle, where a card is removed at random from the deck and then reinserted into the deck at random. Ayyer, Klee, and Schilling \cite{aks} defined the extended promotion operator and studied promotion Markov chains for arbitrary finite posets, of which the Tsetlin library is the case when the poset is the antichain. 

In a different development, the top to random shuffle has also been studied from the viewpoint of the elements $B_a$ of the algebra $\mathbb{Q}[S_n]$, which is the viewpoint we adopt in most of this paper; intuitively, $B_a$ represents all possible outcomes of removing the first $a$ cards of a deck and then reinserting these cards. Diaconis, Fill, and Pitman \cite{diaconistop} derived (for $1 \leq k
\leq n$), using probabilistic arguments, the expansion formula 
\begin{equation}
\label{eqn1}
B_1^k = \sum_{i=1}^n{i^k
\frac{1}{i!}\sum_{a=i}^n{\frac{(-1)^{a-i}}{(a-i)!}B_a}},
\end{equation}
which describes $k$ iterations of the top to random 1-shuffles. Garsia
\cite{garsiatop} later derived (for $1 \leq k \leq n$), using standard combinatorial
manipulations, the expansion formula
\begin{equation}
\label{eqn2}
B_1^k = \sum_{a=1}^k{S_{k,a}B_a},
\end{equation}
from which (\ref{eqn1}) follows; $S_{k,a}$ is the number of $a$-part partitions of the set $[k]$. Garsia then used (\ref{eqn1}) to determine,
among other things, the eigenvalues and eigenspaces of the image of $B_1$ under
the left regular representation.

Among other uses of formula (\ref{eqn2}), one can calculate the number
of ways to obtain a particular arrangement of the deck via $k$ iterations of the
top to random 1-shuffles, as well as the probability of obtaining such
arrangement of the deck. For example, the number of ways to obtain the identity
deck $12 \cdots n$ is the $k$th Bell number $b_k = \sum_{a=1}^k{S_{k,a}}$ by
this formula, since each $B_a$ contains exactly one copy of the identity deck. An expansion formula for $B_{a_1}B_{a_2} \cdots B_{a_k}$ would allow us to do even more calculations of this kind, where we shuffle the first $a_1$ cards of the deck, then the first $a_2$ cards of the resulting deck, and so on; see Example \ref{shuf ways}.

In Section \ref{prelim}, we introduce the notations and conventions
used in this paper.  In Section \ref{arb fin}, we give a bijection which allows us to derive the expansion formula \[B_{a_1}B_{a_2} \cdots B_{a_k} = \sum_{j=\max(a_1,\ldots,a_k)}^{\min(\sum_{m=1}^k{a_m},n)}{|Q_j^{a_1,\ldots,a_k}|B_j},\] where the coefficients can be calculated by the formula \[|Q_j^{a_1,\ldots,a_k}| = \sum_{\substack{\sum_{c=2}^k{l_c}=j-a_1 \\ l_c \in [0,a_c]}}{\left(\prod_{c=2}^k{{a_c \choose l_c}P \left(a_1+\sum_{i=2}^{c-1}{l_i},a_c-l_c \right)}\right)}\] where $P(m,l) := {m \choose l}l!$. As we will see, one combinatorial interpretation of $|Q_j^{a_1,\ldots,a_k}|$ is the number of ways to obtain the identity deck by shuffling the first $a_1$ cards of the deck, then the first $a_2$ cards of the resulting deck, and so on, such that the cards labeled $1, 2, \ldots, j$ are the only ones touched by our hands (removed from the deck and then reinserted) through this sequence of shuffles. In Section \ref{g-perm}, for an arbitrary group $G$ we introduce the group of $G$-permutations $S_n^G$ as the wreath product $G \wr S_n$ and we define elements $\hat{B}_c$ of the algebra $\mathbb{Q}[S_n^G]$ that describe the top to random shuffle of $c$ cards where each card in the deck now has $|G|$ faces, each of which is labeled by an element of $G$. Then we derive the expansion formula \[\hat{B}_{a_1}\hat{B}_{a_2} \cdots \hat{B}_{a_k} = \sum_{c=\max(a_1,\ldots,a_k)}^{\min(\sum_{m=1}^k{a_m},n)}{|Q_c^{a_1,\ldots,a_k}||G|^{(\sum_{i=1}^k{a_k})-c}\hat{B}_c}\] and show how it can be used in calculations; see Example \ref{g-perm calc}. Afterward, we show how the same method can be used to generalize other expansion formulae of $\mathbb{Q}[S_n]$ to $\mathbb{Q}[S_n^G]$. Finally,
in Section \ref{further}, we discuss possible related future directions of research.

% \acknowledgements
% \label{sec:ack}
% I would like to thank my advisor Anne Schilling and my colleague Travis Scrimshaw for reviewing the contents of Section \ref{bijection} and simplifying my proof of Theorem \ref{thm1}.

\section*{Acknowledgements}
I would like to thank my advisor Anne Schilling and my colleague Travis
Scrimshaw for helping simplify my proof of Theorem \ref{genbij}. I would also like to thank Joel Lewis for pointing out that Corollary \ref{fac calc} is an immediate consequence of Lemma \ref{grp elt fac}.

% Section 2
\section{Preliminaries}
\label{prelim}
We fix $n$ throughout this paper, and we follow the notation of \cite{garsiatop}, which we describe briefly here. For words $u = u_1u_2 \cdots u_l$ and $v = v_1v_2 \cdots v_m$, let $u \shuffle v$ denote the sum of all words $w = w_1w_2 \cdots w_{l+m}$ with $\{w_1,w_2,\ldots,w_{l+m}\} = \{u_1,\ldots,u_l,v_1,\ldots,v_m\}$ such that for $i < j$, $p < q$, $u_i = w_{i'}$, $u_j = w_{j'}$, $v_p = w_{p'}$, $v_q = w_{q'}$ we have $i' < j'$ and $p' < q'$; $u \shuffle v$ is called the shuffle product of $u$ and $v$. Recall from Garsia's paper that $B_a$ is the element $B_a = 1 \shuffle 2 \shuffle 3 \shuffle \cdots \shuffle a \shuffle W_{a,n} = \sum_{\alpha \in S_a}{\alpha \shuffle W_{a,n}}$ of the group algebra $\mathbb{Q}[S_n]$, where $W_{a,n}$ is
the word $W_{a,n} = (a+1)(a+2) \cdots n$. The motivation is that we have a deck
of cards labeled $1, 2, \ldots, n$, and $B_a$ represents all possible decks
that may result from removing the cards $1, 2, \ldots, a$ and then inserting
them back into the deck (consisting of the cards $a+1, a+2, \ldots, n$). Each of
these resulting decks can be viewed as a permutation in $S_n$, described more
precisely in the following. 

Such a resulting deck $u = c_1c_2 \cdots c_n$ (where $c_i \in [n]$) can be
viewed as the permutation $\colvec{\id}{u}^{-1}$, where $\id = 12 \cdots n$;
note that $\colvec{\id}{u}^{-1} = \colvec{u''}{\id}$ is the biword
$\colvec{\id}{u}$ sorted lexicographically with priority on the bottom row, and
is the inverse of the permutation $\colvec{\id}{u}$. In other words, $u = c_1c_2
\cdots c_n$ can be viewed as the permutation that maps $c_i \mapsto i$ for all
$i \in [n]$; intuitively, the notation $u = c_1c_2 \cdots c_n$ tells us that
card $c_i$ is moved to position $i$ upon shuffling. This convention of
representing a permutation in $S_n$ by a deck of cards will be called
\textbf{deck notation} (for brevity) in this paper, and is just the inverse one-line notation. Of
all the different notations used
to denote a permutation in $S_n$, the deck notation seems most natural for our
purposes, so we will be using the deck notation throughout this paper.

\begin{fact}
\label{perm as term}
Given a permutation (deck of cards) $\sigma$, let $m_{\sigma}$ be the smallest letter/card such that $m_{\sigma} \cdots (n-1) n$ is a subword of $\sigma$. Then $\sigma$ is a term of $B_c$ for any $c \geq  m_{\sigma}-1$.
\end{fact}

\begin{exm}
The permutation $32145 \cdots n$ is a term of $B_c$ for any $c \geq 2$.
\end{exm}

\begin{fact}
\label{perm injec beg}
In the deck notation, a permutation $\sigma$ is completely determined once we specify the $B_c$ of which $\sigma$ is a term (i.e. specifying that $\sigma$ shuffles the first $c$ cards) and to what positions (integers in $[n]$) in the deck these $c$ cards are sent.
\end{fact}

Finally,
given two permutations $\sigma, \tau$ in deck notation, we will always compute
their product $\sigma\tau$ from left to right; in other words, if $\sigma$ maps
$i \mapsto \sigma(i)$ and $\tau$ maps $i \mapsto \tau(i)$, then $\sigma\tau$
maps $i \mapsto \tau(\sigma(i))$.

\begin{exm}
To compute the product of $\sigma_1 = 2134 \cdots n$ and $\sigma_2 = 2314 \cdots
n$ in deck notation, note that the former maps $1 \mapsto 2$, $2 \mapsto 1$, and
fixes all other cards, while the latter maps $1 \mapsto 3$, $2 \mapsto 1$, $3
\mapsto 2$, and fixes all other cards. Thus, we have $\sigma_1\sigma_2 = (2134
\cdots n)(2314 \cdots n) = 1324 \cdots n$ in deck notation. \\
Here is a more intuitive and visual way of seeing this: Whenever a permutation $\sigma$ acts on a card $m$ (by moving or fixing it), we say that $\sigma$ \textbf{hits} card $m$. Both $\sigma_1$ and $\sigma_2$ can be viewed as terms of $B_1$; in other words, they move or fix the top card of the deck.
Start with the default deck $123 \cdots n$. $\sigma_1$ hits card $1$ (the top card) and sends it to
position $2$, giving the deck $2134 \cdots n$. Then $\sigma_2$ hits card $2$
(now the top card) and sends it to position $3$, giving the deck $1324 \cdots
n$, which is also the desired product. Since only the cards $1, 2$ have been hit in $\sigma_1\sigma_2$, we can treat $1324 \cdots n$ as a term of $B_2$.
\end{exm}

%In deriving an expansion formula \[B_{a_1}B_{a_2} \cdots B_{a_k} = \sum_{j=\max(a_1,\ldots,a_k)}^{\min(\sum_{m=1}^k{a_m},n)}{|Q_j^{a_1,\ldots,a_k}|B_j}\] Our overarching theme is putting a_1 + \ldots + a_k elements into j bins, thereby creating a j-part partition of [a_1 + \ldots a_k].

We now generalize this visual approach to the multiplication of permutations and apply it in Section \ref{arb fin}. A term of $B_1$, as we have seen above, hits the first card of the deck. In general, a term $\tau$ of $B_c$ hits (simultaneously) the first $c$ cards of the deck; we shall call $\tau$ a \textbf{$c$-shuffle} in this case. By Fact \ref{perm injec beg}, we can write $\tau$ as the $c$-tuple $(\tau(1),\tau(2),\ldots,\tau(c))$ where $\tau(i) \in [n]$ is the position in the deck to which card $i$ is sent by $\tau$; in this way, we can also view $\tau$ as an injection from $[c]$ to $[n]$. We will use this convention for permutations in what follows.

% Section 3
\section{Multiplication Formula for Arbitrary Finite Products}
\label{arb fin}
We find an expansion formula for the $k$-fold product $B_{a_1}B_{a_2} \cdots B_{a_k}$ in terms of the elements $B_c$.
\subsection{A Bijection between Shuffles and Set Partitions}
\label{shuffles partitions}
Notice that $B_{a_1}B_{a_2} \cdots B_{a_k}$ is a sum of terms $\sigma_1\sigma_2 \cdots \sigma_k$ where
$\sigma_i \in S_n$ is a term of $B_{a_i}$. Here
we regard two terms $\sigma'_1\sigma'_2 \cdots \sigma'_k$, $\sigma''_1\sigma''_2
\cdots \sigma''_k$ of $B_{a_1}B_{a_2} \cdots B_{a_k}$ as distinct if $(\sigma'_1,\sigma'_2,\ldots,\sigma'_k)
\neq (\sigma''_1,\sigma''_2,\ldots,\sigma''_k)$, even if $\sigma'_1\sigma'_2 \cdots
\sigma'_k$, $\sigma''_1\sigma''_2 \cdots \sigma''_k$ are equal as products in $S_n$. We denote $\sigma_i = (\sigma_i(1),\sigma_i(2),\ldots,\sigma_i(a_i))$ where $\sigma_i(m)$ is the position to which $\sigma_i$ sends the $m$th card of the deck. Thus, the sequence $(\sigma_1,\ldots,\sigma_k)$ gives rise to the $(\sum_{j=1}^k{a_j})$-tuple $(\sigma_1(1),\sigma_1(2),\ldots,\sigma_1(a_i),\sigma_2(1), \break \sigma_2(2),\ldots,\sigma_2(a_2),\ldots,\sigma_k(1),\sigma_k(2),\ldots,\sigma_k(a_k))$, where $\sigma_i$ is the $i$th segment of this tuple. Now, for ease of counting, we relabel this $(\sum_{j=1}^k{a_j})$-tuple as $(b_1,b_2,\ldots,b_{\sum_{j=1}^k{a_j}})$. In other words, the symbol $b_{m+\sum_{j=1}^{i-1}{a_j}}$ will denote the position $\sigma_i(m)$. Determining $(\sigma_1,\ldots,\sigma_k)$ is thus the same as determining $(b_1,b_2,\ldots,b_{\sum_{j=1}^k{a_j}}) = (\sigma_1(1),\sigma_1(2),\ldots,\sigma_1(a_i),\sigma_2(1), \sigma_2(2),\ldots,\sigma_2(a_2),\ldots,\sigma_k(1), \break \sigma_k(2),\ldots,\sigma_k(a_k))$.

Since $\sigma_i$ acts on the $m$th card by sending it to position $\sigma_i(m) = b_{m+\sum_{j=1}^{i-1}{a_j}}$, we say that the $m$th card is \textbf{hit} by $b_{m+\sum_{j=1}^{i-1}{a_j}}$ through $\sigma_i$, and we will call the entries $b_1,b_2,\ldots,b_{\sum_{j=1}^k{a_j}}$ the $\textbf{hitters}$ of $\sigma_1\sigma_2 \cdots \sigma_k$; here we view each $b_l$ as a formal symbol rather than an integer, and this will be the case whenever we talk about hitting. We say that cards $1, 2, \ldots, j$ are \textbf{hit} by
$\sigma_1 \cdots \sigma_k$ if each of them is hit by some $b_l$. Thus, each such
term $\sigma_1 \sigma_2 \cdots \sigma_k$ hits exactly the cards $1, 2, \ldots, j$ for a
unique $j \in [\max(a_1,\ldots,a_k),\min(\sum_{m=1}^k{a_m},n)]$; it must hit at least the cards $1, 2, \ldots, \max(a_1,\ldots,a_k)$, and it cannot hit more than $n$ cards.
\begin{rem}
We consider every term of $B_c$ to be hitting the first $c$ cards. Only these cards are touched by our hands and then reinserted. The other $n-c$ cards move in the shuffling as well, but they are not touched by our hands; they move only as a result of being displaced by these reinserted cards. Thus, a real-life interpretation of ``the term $\sigma_1 \sigma_2 \cdots \sigma_k$ hits exactly the cards $1, 2, \ldots, j$'' is that our hands touch exactly the cards $1, 2, \ldots, j$ as we carry out the sequence of shuffles $\sigma_1, \sigma_2, \ldots, \sigma_k$ in that order.
\end{rem}

Partition the terms of $B_{a_1}B_{a_2} \cdots B_{a_k}$ into the sets $D_j = \{$terms $(b_1,b_2,\ldots,b_{\sum_{m=1}^k{a_m}}) \break = (\sigma_1,\ldots,\sigma_k) | \sigma_1 \cdots \sigma_k$ hits only the cards $1, 2, \ldots, j\}$, for $j = \max(a_1,\ldots,a_k), \max(a_1,\ldots,a_k)+1, \ldots, \min(\sum_{m=1}^k{a_m},n)$. Now notice that each $\sigma_1 \cdots \sigma_k \in D_j$ is equal (as a permutation) to a term of
$B_j$. Thus, we can partition $D_j$ into the sets $D_{j,t} = \{(\sigma_1,\sigma_2,\ldots,\sigma_k) \in D_j | \sigma_1 \cdots \sigma_k = t$ as
permutations$\}$ for $t$ a term of $B_j$.

Let $P_j([l])$ denote the set of all partitions of $[l]$ into exactly $j$ nonempty parts. Order the parts of each $j$-part partition in increasing order by the smallest element of each part. In other words, if $\alpha \in P_j([l])$, then we will always write $\alpha = \{\alpha_1,\alpha_2,\ldots,\alpha_j\}$ where $\min
\alpha_1 < \min \alpha_2 < \ldots < \min \alpha_j$. Let $S_{l,j} := |P_j([l])|$, which is called the Stirling number of the second kind.
\begin{exm}
$P_2([3])$ consists of the partitions $\{\{1\},\{2,3\}\}$, $\{\{1,2\},\{3\}\}$, $\{\{1,3\},\{2\}\}$. We have $S_{3,2} = 3$.
\end{exm}

To derive an expansion formula for $B_{a_1}B_{a_2} \cdots B_{a_k}$, we will set up a bijection $\phi_j$ between $D_{j,t}$ and a certain subset of $P_j([\sum_{m=1}^k{a_m}])$ that we will denote by $Q_j^{a_1,\ldots,a_k}$, for any term $t$ of $B_j$; we will later see that $\phi_j$ does not depend on $t$, but its inverse does. 

To motivate our definition of $Q_j^{a_1,\ldots,a_k} \subset P_j([\sum_{m=1}^k{a_m}])$, consider the way $(b_1,b_2,\ldots,\break b_{\sum_{m=1}^k{a_m}}) \in D_{j,t}$ acts on the default deck $12 \cdots n$. After the cards $1, 2, \ldots, a_1$ are hit by $b_1, b_2, \ldots, b_{a_1}$, the $b_{a_1+1}, b_{a_1+2}, \ldots, b_{a_1+a_2}$ must hit $a_2$ cards of the cards $1, 2, \ldots, a_1, a_1+1, \ldots, a_1+a_2$, subject to the condition that for $b_{r_{a_1+1}}, b_{r_{a_1+2}}, \ldots, b_{r_l}$ (where $l \in [a_1+1,a_1+a_2]$) hitting the cards $a_1+1, a_1+2, \ldots, l$ respectively, we must have $r_{a_1+1} < r_{a_1+2} < \ldots < r_l$; in other words, the cards $a_1+1, a_1+2, \ldots, l$ must be hit in this order. Generally, after the cards $1, 2, \ldots, i$ are hit by $b_1, b_2, \ldots, b_{1+\sum_{m=1}^{c-1}{a_m}}, b_{2+\sum_{m=1}^{c-1}{a_m}}, \ldots, b_{\sum_{m=1}^c{a_m}}$, the next $a_{c+1}$ cards hit must be from the cards $1, 2, \ldots, i, i+1, \ldots, \min(i+a_{c+1},n)$, subject to the condition that for $b_{r_{i+1}}, b_{r_{i+2}}, \ldots, b_{r_{l}}$ (where $l \in [i+1,\min(i+a_{c+1},n)]$) hitting the cards $i+1, i+2, \ldots, l$ respectively, we must have $r_{i+1} < r_{i+2} < \ldots < r_l$; in other words, the cards $i+1, i+2, \ldots, l$ must be hit in this order.

Based on the above observation, we now describe precisely the relevant subset $Q_j^{a_1,\ldots,a_k} \subset P_j([\sum_{m=1}^k{a_m}])$ in the following manner:
To construct $\alpha = \{\alpha_1,\alpha_2,\ldots,\alpha_j\} \in Q_j^{a_1,\ldots,a_k}$, put the elements $1, 2, \ldots, \sum_{m=1}^k{a_m}$ of $[\sum_{m=1}^k{a_m}]$ (in that order) surjectively into the (initially empty) bins (``parts'') $\alpha_1, \alpha_2, \ldots, \alpha_j$, subject to the following \textbf{Rules}. 
\begin{enumerate}
\item For each $1 \leq i \leq k$ the elements $1+\sum_{m=1}^{i-1}{a_m}, 2+\sum_{m=1}^{i-1}{a_m}, \ldots, \sum_{m=1}^i{a_m}$ must be put injectively into $a_i$ of these bins. At the outset, the elements $1, 2, \ldots, a_1$ must be put injectively into $\alpha_1, \alpha_2, \ldots, \alpha_{a_1}$ respectively.
\item Let $a_1 \leq i \leq j$ be such that $\alpha_1, \alpha_2, \ldots, \alpha_i$ are the only bins filled with at least one element just before placing the elements $1+\sum_{m=1}^{c-1}{a_m}, 2+\sum_{m=1}^{c-1}{a_m}, \ldots, \sum_{m=1}^{c}{a_m}$ where $c \in [2,k]$. (In this case, we have $i \leq \sum_{m=1}^{c-1}{a_m}$.) We place the elements $1+\sum_{m=1}^{c-1}{a_m}, 2+\sum_{m=1}^{c-1}{a_m}, \ldots, \sum_{m=1}^{c}{a_m}$ in this manner:
Pick $l_c \in [0,\min(a_{c},j-i)]$ and elements $r_{1} < r_{2} < \ldots < r_{l_c} \in [1+\sum_{m=1}^{c-1}{a_m},\sum_{m=1}^{c}{a_m}]$. Put element $r_u$ into bin $\alpha_{i+u}$. Then put the remaining elements of $[1+\sum_{m=1}^{c-1}{a_m},\sum_{m=1}^{c}{a_m}]$ injectively into any $a_{c}-l_c$ of the bins $\alpha_1, \ldots, \alpha_i$.
\item We must have $a_1 + \sum_{c=2}^{k}{l_c} = j$.
\end{enumerate}
We call the resulting element $\alpha = \{\alpha_1,\alpha_2,\ldots,\alpha_j\}$ of $Q_j^{a_1,\ldots,a_k}$ an $(a_1,\ldots,a_k)$-segmented $j$-part partition of $[\sum_{m=1}^k{a_m}]$.

We claim that
\begin{thm}
\label{exp'}
\[B_{a_1}B_{a_2} \cdots B_{a_k} = \sum_{j=\max(a_1,\ldots,a_k)}^{\min(\sum_{m=1}^k{a_m},n)}{|Q_j^{a_1,\ldots,a_k}|B_j}\].
\end{thm}

We first need the following theorem.
\begin{thm}
\label{genbij}
Let $j \in [\max(a_1,\ldots,a_k),\min(\sum_{m=1}^k{a_m},n)]$ and let $t$ be any term of $B_j$. Define the map $\phi_j:
D_{j,t} \rightarrow Q_j^{a_1,\ldots,a_k}$ given by $(b_1,b_2,\ldots,b_{\sum_{m=1}^k{a_m}}) = (\sigma_1,\sigma_2,\ldots,\sigma_k) \mapsto \alpha = \{\alpha_1,\alpha_2,\ldots,\alpha_j\}$ where $\alpha_i = \{l \in [\sum_{m=1}^k{a_m}] | $card $i$ is hit by $b_l\}$. Then $\phi_j$ is a bijection.
\end{thm}
\begin{proof}
Suppose $(b_1,b_2,\ldots,b_{\sum_{m=1}^k{a_m}}) = (\sigma_1,\sigma_2,\ldots,\sigma_k) \in D_{j,t}$. Then the $b_1,b_2,\ldots,b_{\sum_{m=1}^k{a_m}}$ hit (in that order) the cards $1, 2, \ldots, j$, which corresponds to putting the elements $1, 2, \ldots, \break \sum_{m=1}^k{a_m}$ of $[\sum_{m=1}^k{a_m}]$ (in that order) into the bins $\alpha_1, \alpha_2, \ldots, \alpha_j$ subject to Rules 1-3 above. This shows that $\phi_j(\sigma_1,\sigma_2,\ldots,\sigma_k)$ is indeed an element of $Q_j^{a_1,\ldots,a_k}$.

We now check simultaneously the injectivity and surjectivity of $\phi_j$. Let $\alpha = \{\alpha_1,\alpha_2,\ldots,\alpha_j\} \in Q_j^{a_1,\ldots,a_k}$. We show that $\alpha$ uniquely determines the $(b_1,b_2,\ldots,b_{\sum_{m=1}^k{a_m}}) = (\sigma_1,\sigma_2,\ldots,\sigma_k)$ for which $\phi_j(\sigma_1,\sigma_2,\ldots,\sigma_k) = \alpha$. We may assume that the parts of $\alpha$ have been ordered lexicographically; in other words, $\min \alpha_l < \min \alpha_m$ for $1 \leq l < m \leq j$. 

We will determine what the $\sigma_k, \sigma_{k-1}, \ldots, \sigma_1$ (or equivalently $b_{\sum_{m=1}^k{a_m}},\ldots,b_2,b_1$) are, in that order. For each $b_l$ let $\alpha_{r_l}$ denote the part of $\alpha$ containing $l$; in other words, $r_l$ is the card hit by $b_l$. If $p_{r_{\sum_{m=1}^k{a_m}}}$ is the position of card $r_{\sum_{m=1}^k{a_m}}$ in $t_k := t$, then we know that $b_{\sum_{m=1}^k{a_m}}$ sends the $a_k$th card to position $p_{r_{\sum_{m=1}^k{a_m}}}$. If $p_{r_{(\sum_{m=1}^k{a_m})-1}}$ is the position of card $r_{(\sum_{m=1}^k{a_m})-1}$ in $t_k$, then we know that $b_{(\sum_{m=1}^k{a_m})-1}$ sends the $(a_{k}-1)$st card to position $p_{r_{(\sum_{m=1}^k{a_m})-1}}$. Continuing this process, we determine that $\sigma_k = (b_{(\sum_{m=1}^{k-1}{a_m})+1},\ldots,b_{(\sum_{m=1}^k{a_m})-1},b_{\sum_{m=1}^k{a_m}})$, where $b_{(\sum_{m=1}^k{a_m})-u}$ sends the $(a_k-u)$th card to position $p_{r_{(\sum_{m=1}^k{a_m})-u}}$. Notice that, by Rule 1 above, $\alpha_{r_u} \neq \alpha_{r_v}$ whenever $(\sum_{m=1}^{k-1}{a_m})+1 \leq u,v \leq \sum_{m=1}^k{a_m}$ are distinct, so $r_u \neq r_v$ whenever $(\sum_{m=1}^{k-1}{a_m})+1 \leq u,v \leq \sum_{m=1}^k{a_m}$ are distinct. Thus, $\sigma_k = (b_{(\sum_{m=1}^{k-1}{a_m})+1},\ldots,b_{(\sum_{m=1}^k{a_m})-1}, \break b_{\sum_{m=1}^k{a_m}})$ is indeed a term of $B_{a_k}$. Next, we can similarly determine the $b_{\sum_{m=1}^{k-1}{a_m}}, b_{(\sum_{m=1}^{k-1}{a_m})-1}, \break \ldots, b_{(\sum_{m=1}^{k-2}{a_m})+1}$ by looking at $t_{k-1}$, the deck right before being hit by
$\sigma_k$; in other words, $t_{k-1}$ is obtained from $t_k$ by reversing the action of $\sigma_k$. In general, fix any $0 \leq i \leq k$, and let $t_i$ be the deck right before being hit by $\sigma_{i+1}$. If $p_{r_{(\sum_{m=1}^i{a_m})-u}}$ is the position of card $r_{(\sum_{m=1}^i{a_m})-u}$ in $t_i$, then we know that $b_{(\sum_{m=1}^i{a_m})-u}$ sends the $(a_i-u)$th card to position $p_{r_{(\sum_{m=1}^i{a_m})-u}}$. In this way, $\sigma_i$ is uniquely determined, and $\sigma_i$ is indeed a term of $B_{a_i}$ by Rule 1 above.

Now we must check that, as we reverse the actions of $\sigma_k, \sigma_{k-1}, \ldots, \sigma_1$, we eventually get the identity deck $12 \cdots n$; in other words, we must check that $t_0 = 12 \cdots n$.

Let $m_1 < m_2 \in [j]$. We now show that card $m_1$ eventually ends up to the left of card $m_2$ as we reverse the actions of $\sigma_k, \sigma_{k-1}, \ldots, \sigma_1$ in that order. Consider the elements $\min \alpha_{m_1}, \min \alpha_{m_2} \in [\sum_{m=1}^k{a_m}]$, because the actions of $b_{\min \alpha_{m_1}}$ and $b_{\min \alpha_{m_2}}$ are the last to be reversed on the cards $m_1$ and $m_2$, respectively. Suppose that both $\min \alpha_{m_1}, \min \alpha_{m_2}$ belong in the interval $[1+\sum_{m=1}^{i-1}{a_m},\sum_{m=1}^i{a_m}]$ for some $i \in [k]$. Then by Rule 2 above, we must have $\min \alpha_{m_1} < \min \alpha_{m_2}$ (since $\alpha_{m_1}$ must receive its first element before $\alpha_{m_2}$ does), so card $m_1$ will indeed end up to the left of card $m_2$ after the action of $\sigma_i$ is reversed. Suppose that $\min \alpha_{m_1} \in [1+\sum_{m=1}^{c-1}{a_m},\sum_{m=1}^c{a_m}]$ and $\min \alpha_{m_2} \in [1+\sum_{m=1}^{d-1}{a_m},\sum_{m=1}^d{a_m}]$ for some $c \neq d \in [k]$. This means $b_{\min \alpha_{m_1}}$ belongs to $\sigma_c$ and $b_{\min \alpha_{m_2}}$ belongs to $\sigma_d$. By Rule 2, we must have $c < d$, because $\alpha_{m_1}$ must receive its first element before $\alpha_{m_2}$ does. It follows that the action of $\sigma_d$ gets reversed before the action of $\sigma_c$, so card $m_1$ must indeed end up to the left of card $m_2$.
% Note: The contiguousness of the hitters in each interval is a VERY important property for our purposes.
\end{proof}
\begin{rem}
We see from this proof that $\phi_j^{-1}$ is given as follows. Given $\alpha = \{\alpha_1,\alpha_2,\ldots,\alpha_j\} \in Q_j^{a_1,\ldots,a_k}$, for each $b_l$ let $\alpha_{r_l}$ denote the part of $\alpha$ containing $l$; in other words, $r_l$ is the card hit by $b_l$. For each $0 \leq i \leq k$ let $t_i$ denote the deck right before being hit by $\sigma_{i+1}$, and let $p_{i,c}$ denote the position of card $r_c$ in $t_i$. Then $\phi_j^{-1}(\alpha) = (b_1,b_2,\ldots,b_{\sum_{m=1}^k{a_m}})$ where $b_{d+\sum_{m=1}^{i-1}{a_m}}$ sends the $d$th card in the deck to position $p_{i,d+\sum_{m=1}^{i-1}{a_m}}$, for $d \in [a_i]$.
\end{rem}

By definition, the sets $D_j$ partition the set of terms of $B_{a_1}B_{a_2} \cdots B_{a_k}$, and the sets $D_{j,t}$ partition $D_j$. Each set $D_j$ forms $|Q_j^{a_1,\ldots,a_k}|$ copies of $B_j$, as the sets $D_{j,t}$ are all equinumerous with cardinality $|Q_j^{a_1,\ldots,a_k}|$ by Theorem \ref{genbij}. Therefore, Theorem \ref{exp'} follows as a corollary.

\subsection{Calculating the Coefficients $|Q_j^{a_1,\ldots,a_k}|$}
We now find an explicit formula for $|Q_j^{a_1,\ldots,a_k}|$, which is the number of ways of putting the elements of $[\sum_{m=1}^k{a_m}]$ (in the order $1, 2, \ldots, \sum_{m=1}^k{a_m}$) into the (initially empty) containers $\alpha_1, \alpha_2, \ldots, \alpha_j$, subject to the Rules 1-3 we described earlier. Here we use the notation $P(m,l) = {m \choose l}l!$.

As a start, we can partition the set $Q_j^{a_1,\ldots,a_k}$ by which elements of $[\sum_{m=1}^k{a_m}]$ we choose to be the $\min \alpha_1$, $\min \alpha_2$, $\ldots$, $\min \alpha_j$; by definition $\min \alpha_l$ is the first element put into $\alpha_l$. Of course, we must have $\min \alpha_i = i$ for $i = 1, 2, \ldots, a_1$. It then remains to determine the $\min \alpha_{a_1+1}, \min \alpha_{a_1+2}, \ldots, \min \alpha_j$ among the remaining $\sum_{m=2}^k{a_m}$ elements; we will call these $j-a_1$ elements \textbf{anchor elements} or \textbf{anchors}.

Once these $j-a_1$ elements (equivalently, $j$ elements) have been chosen, then we
are guaranteed that the eventual partition $\alpha = \{\alpha_1, \alpha_2,
\ldots, \alpha_j\}$ will indeed have $j$ parts. By Rules 2 and 3, for each interval $[1+\sum_{m=1}^{c-1}{a_m},\sum_{m=1}^c{a_m}]$ (where $c \in [2,k]$) we need to choose certain elements $m_1 < m_2 < \ldots < m_{l_c} \in [1+\sum_{m=1}^{c-1}{a_m},\sum_{m=1}^c{a_m}]$ (where $l_c \leq a_c$ can be zero) to be the anchors, so that in the end we have $\sum_{c=2}^k{l_c}=j-a_1$. Fix such a tuple $(l_2,l_3,\ldots,l_k)$.

Consider the interval $[1+\sum_{m=1}^{c-1}{a_m},\sum_{m=1}^c{a_m}]$ where $c \in [2,k]$. There are ${a_c \choose l_c}$ ways to choose the anchor elements $m_1, m_2, \ldots, m_{l_c}$ in this interval. Once the elements $m_1, m_2, \ldots, m_{l_c}$ have been chosen as the
anchors, the remaining $a_c-l_c$ elements of this interval have to be put injectively into the sets $\alpha_1, \alpha_2, \ldots,
\alpha_{a_1+\sum_{i=2}^{c-1}{l_i}}$ (because only these sets have received at least one element at this time), and there are $P(a_1+\sum_{i=2}^{c-1}{l_i},a_c-l_c)$ ways to do this. Thus, there are a total of $N_c := {a_c \choose l_c}P(a_1+\sum_{i=2}^{c-1}{l_i},a_c-l_c)$ ways to assign the elements of this interval.

Notice that all the intervals $[1+\sum_{m=1}^{c-1}{a_m},\sum_{m=1}^c{a_m}]$ where $c \in [2,k]$ are disjoint. Hence there are a total of $\prod_{c=2}^k{N_c} = \prod_{c=2}^k{{a_c \choose l_c}P(a_1+\sum_{i=2}^{c-1}{l_i},a_c-l_c)}$ ways to assign all the elements, for each
tuple $(l_2,l_3,\ldots,l_k)$.

We conclude that (for $k \geq 2$) the formula is 
\begin{equation}\label{cardform} |Q_j^{a_1,\ldots,a_k}| = \sum_{\substack{\sum_{c=2}^k{l_c}=j-a_1 \\ l_c \in [0,a_c]}}{\left(\prod_{c=2}^k{{a_c \choose l_c}P \left(a_1+\sum_{i=2}^{c-1}{l_i},a_c-l_c \right)}\right)}, \end{equation}
where the sum is over all tuples $(l_2,l_3,\ldots,l_k)$ (with $l_c \in [0,a_c]$) such that $\sum_{c=2}^k{l_c}=j-a_1$.
\begin{rem}
For any tuple $(l_2,l_3,\ldots,l_k)$ (with $l_c \in [0,a_c]$) such that $\sum_{c=2}^k{l_c}=j-a_1$, the nonnegative term \[\prod_{c=2}^k{{a_c \choose l_c}P \left(a_1+\sum_{i=2}^{c-1}{l_i},a_c-l_c \right)}\] counts the number of ways of ways to obtain the identity deck through this sequence of $k$ top to random shuffles such that only the cards $1, 2, \ldots, j$ (cumulatively) are hit and exactly $l_c$ new cards are hit during the $c$th shuffle (where a total of $a_c$ cards, old and new, are hit).
\end{rem}

\subsection{Applications}
In particular, for $a_1 = \ldots = a_k = 1$, Theorem \ref{exp'} gives us \[B_1^k = \sum_{j=1}^{\min(k,n)}{|Q_j^{1,\ldots,1}|B_j}.\] By the above and the following proposition, we see that Garsia's formula (\ref{eqn2}) follows as a special case, except now the restriction $1 \leq k \leq n$ is removed.
\begin{prop}
$|Q_j^{1,\ldots,1}| = S_{k,j}$.
\end{prop}
\begin{proof}
Looking at Rules 1-3 for the construction of a segmented partition, we see that each $l_c$ is $0$ or $1$ for $c \in [2,k]$. If $l_c = 1$, then the element $c$ is an anchor and is placed into an empty bin; the element $1$ is always placed into the bin $\alpha_1$. If $l_c = 0$, then the element $c$ is placed into a nonempty bin. Every $j$-part partition $\alpha \in P_j([k])$ (whose parts are ordered by their smallest elements) can be constructed in this way, so the claim follows.
\end{proof}
\begin{cor}
By taking $a_1 = \ldots = a_k = 1$ in (\ref{cardform}), we obtain another formula for the Stirling number $S_{k,j}$.
\end{cor}

For $k = 2$, Theorem \ref{exp'} gives us \[B_{a_1}B_{a_2} = \sum_{j=\max(a_1,a_2)}^{\min(a_1+a_2,n)}{|Q_j^{a_1,a_2}|B_j}.\] We have $|Q_j^{a_1,a_2}| = \sum_{l_2=j-a_1}{(\prod_{c=2}^2{{a_c \choose l_c}P(a_1+\sum_{i=2}^{c-1}{l_i},a_c-l_c)})} = {a_2 \choose l_2}P(a_1,a_2-l_2) = \break {a_2 \choose j-a_1}P(a_1,a_2+a_1-j) = \frac{a_2!}{(j-a_1)!(a_2+a_1-j)!} \frac{a_1!}{(j-a_2)!}$. This gives Garsia's formula (1.3) \[B_{a_1}B_{a_2} = \sum_{j=\max(a_1,a_2)}^{\min(a_1+a_2,n)}{\frac{a_2!}{(j-a_1)!(a_2+a_1-j)!} \frac{a_1!}{(j-a_2)!}B_j}\] in \cite{garsiatop}, which also appears in Theorem 4.2 of \cite{diaconistop}.

Now we show how Theorem \ref{exp'} can be used in some calculations.
\begin{exm}
\label{shuf ways}
How many ways can we obtain the deck $i(i-1) \cdots 1(i+1) \cdots n$ via top-to-random shuffling of $a_1$ cards, $a_2$ cards, $\ldots$, $a_k$ cards in that order? This question can be answered by looking at the product $B_{a_1}B_{a_2} \cdots B_{a_k}$, and counting all copies of the term $i(i-1) \cdots 1(i+1) \cdots n$ appearing on the right hand side of Theorem \ref{exp'}. By Fact \ref{perm as term}, $i(i-1) \cdots 1(i+1) \cdots n$ is a term of $B_c$ for any $c \geq i-1$, and each such $B_c$ contains exactly one copy of $i(i-1) \cdots 1(i+1) \cdots n$. It follows that there are \[\sum_{j=i-1}^{\min(\sum_{m=1}^{k}{a_n},n)}{|Q_j^{a_1,\ldots,a_k}|} = \sum_{j = \max(i-1,\max(a_1,\ldots,a_k))}^{\min(\sum_{m=1}^{k}{a_n},n)}{|Q_j^{a_1,\ldots,a_k}|}\] ways to do this. Note that $|Q_j^{a_1,\ldots,a_k}|$ is the number of ways the deck $i(i-1) \cdots 1(i+1) \cdots n$ can be obtained via this sequence of $k$ top to random shuffles through which exactly the cards $1, 2, \ldots, j$ (cumulatively) are touched/picked (or ``hit'', in our terminology) and then reinserted. \\
Since $B_{a_1}B_{a_2} \cdots B_{a_k}$ has $P(n,a_1)P(n,a_2) \cdots P(n,a_k)$ terms/decks in total, the probability of obtaining the deck $i(i-1) \cdots 1(i+1) \cdots n$ via top-to-random shuffling of $a_1$ cards, $a_2$ cards, $\ldots$, $a_k$ cards in that order is \[\frac{1}{\prod_{l=1}^{k}{P(n,a_l)}}\sum_{j = \max(i-1,\max(a_1,\ldots,a_k))}^{\min(\sum_{m=1}^{k}{a_n},n)}{|Q_j^{a_1,\ldots,a_k}|}.\]
Finally, notice that Theorem \ref{exp'} and Fact \ref{perm as term} imply that the probabilities of obtaining decks $\tau_1, \tau_2$ via top-to-random shuffling of $a_1$ cards, $a_2$ cards, $\ldots$, $a_k$ cards in that order are equal whenever $m_{\tau_1} = m_{\tau_2}$.
\end{exm}

% Section 4
\section{Generalizations to the Algebra of G-Permutations}
\label{g-perm}

\subsection{Description of $G$-permutations}
For an arbitrary group $G$ we can define the group $S_n^G$ of \textbf{$G$-permutations} by the wreath product $S_n^G := G \wr S_n$. Now we give a more combinatorial view of $S_n^G$. 

Let $[\hat{n}]$ denote the alphabet $\{\hat{1},\hat{2},\ldots,\hat{n}\}$ (we will use this ``hat'' notation to distinguish these letters from ordinary integers), and let $G^{[\hat{n}]} := \{(g_1,\hat{1}),(g_2,\hat{2}),\ldots,(g_n,\hat{n}) | (g_1,g_2,\ldots,g_n) \linebreak[1] \in G^n\}$. As in the previous sections, we will interpret $G^{[\hat{n}]}$ and $S_n^G$ in terms of decks of $n$ cards, except now each card in the deck has $|G|$ faces, each of which is indexed by an element of $G$.

We denote the identity element by $e$. Also, we define the absolute value of $(g_i,\hat{i})$ to be  $\mathrm{abs}(g_i,\hat{i}) = (e,\hat{i})$. For $g \in G$ we interpret $(g,\hat{i})$ to be card $\hat{i}$ with face $g$ up; we will use $\hat{i}$ as a shorthand for $(e,\hat{i})$ and $g\hat{i}$ as a shorthand for $(g,\hat{i})$ when no confusion arises. The default deck that we start with is $(e,\hat{1})(e,\hat{2})\cdots(e,\hat{n})$, or simply $\hat{1}\hat{2}\cdots\hat{n}$ for short; here we treat each $(e,\hat{j})$ as a single letter. We treat each $\tau \in S_n^G$ as a permutation of $G^{[\hat{n}]}$, and as before we will employ the \textit{deck notation} (inverse one-line notation) for $\tau$. In other words, we write $\tau = [b_1,b_2,\ldots,b_n]$ or simply as a word $\tau = b_1 b_2\cdots b_n$ (where each $b_i \in G^{[\hat{n}]}$ and $\{\mathrm{abs}(b_i) | i \in [n]\} = [\hat{n}]$), where $\tau(b_i) = (e,\hat{i})$; we interpret this as saying $\tau$ sends $b_i$ to position $(e,\hat{i})$ or simply position $i$ in the deck. Then $\tau$ can be seen as a deck obtained by shuffling around the cards of the deck $\hat{1}\hat{2}\cdots\hat{n}$ and simultaneously turning each card so that some face is up; each $b_i$ is equal to some $g_j\hat{j}$ where $g_j \in G$, indicating that card $\hat{j}$ is sent to position $i$ with face $g_j$ up.
\begin{exm}
$\hat{2}(g_1\hat{1})(g_3\hat{3})\hat{4}\cdots\hat{n}$ (where $g_1, g_3 \in G$) is a $G$-permutation that acts on the default deck by sending card $\hat{1}$ to position $2$ with face $g_1$ up, card $\hat{2}$ to position $1$, and card $\hat{3}$ to position $3$ with face $g_3$ up, while fixing all the other $n-3$ cards.
\end{exm}

To describe clearly the action and multiplication of $G$-permutations, let $\sigma, \tau \in S_n^G$ with $\sigma(e,\hat{i}) = (g_{(\sigma,i)},\widehat{p_{(\sigma,i)}})$ and $\tau(e,\hat{j}) = (g_{(\tau,j)},\widehat{p_{(\tau,j)}})$ where $p_{(\sigma,i)}, p_{(\tau,j)} \in [n]$ and $g_{(\sigma,i)}, g_{(\tau,i)} \in G$. For $g \in G$, we have \[\sigma(g,\hat{i}) = g\sigma(e,\hat{i})\] or simply \[\sigma(g\hat{i}) = g\sigma(\hat{i}).\] In our deck notation the product of $\sigma, \tau$ will be carried out from left to right. In other words:
\begin{align*}
\sigma\tau(e,\hat{m}) &= \tau(\sigma(e,\hat{m})) \\
&= \tau(g_{(\sigma,m)},\widehat{p_{(\sigma,m)}}) \\
&=  g_{(\sigma,m)}\tau(e,\widehat{p_{(\sigma,m)}}) \\
&= (g_{(\sigma,m)}g_{(\tau,p_{(\sigma,m)})},\widehat{p_{(\tau,p_{(\sigma,m)})}})
\end{align*} or simply
\begin{align*}
\sigma\tau(\hat{m}) &= \tau(\sigma(\hat{m})) \\
&= \tau(g_{(\sigma,m)}\widehat{p_{(\sigma,m)}}) \\
&= g_{(\sigma,m)}\tau(\widehat{p_{(\sigma,m)}}) \\
&= g_{(\sigma,m)}g_{(\tau,p_{(\sigma,m)})}\widehat{p_{(\tau,p_{(\sigma,m)})}} \numberthis \label{face mult}
\end{align*}
We now interpret the above product in terms of a deck of cards; we will use this interpretation for the rest of this section. Given that $\sigma$ sends card $\hat{i}$ to position $p_{(\sigma,i)}$ with face $g_{(\sigma,i)}$ up and $\tau$ sends card $\hat{j}$ to position $p_{(\tau,j)}$ with face $g_{(\tau,j)}$ up, then $\sigma\tau$ sends card $\hat{m}$ to position $p_{(\tau,p_{(\sigma,m)})}$ with face $g_{(\sigma,m)}g_{(\tau,p_{(\sigma,m)})}$ up.

\begin{exm}
In the case $G = \mathbb{Z}/m\mathbb{Z}$, we get the group $(\mathbb{Z}/m\mathbb{Z}) \wr S_n$ of colored permutations, which are called signed permutations in the case $m=2$. Let $\zeta$ be the primitive $m$th root of unity. Then a $(\mathbb{Z}/m\mathbb{Z})$-permutation is a permutation $\sigma = b_1 b_2 \cdots b_n$ of $\{\zeta^{i_1}\hat{1},\zeta^{i_2}\hat{2},\ldots,\zeta^{i_n}\hat{n} | (i_1,i_2,\ldots,i_n) \break \in (\mathbb{Z}/m\mathbb{Z})^n\} \subset \mathbb{C}$ such that $\sigma(\zeta^{j}b_l) = \zeta^j\sigma(b_l) = \zeta^j\hat{l}$. $\sigma$ can be thought of as a deck of $n$ cards where each card has $m$ faces.
\end{exm}

For $a \in [n]$ and $B_a = \hat{1} \shuffle \hat{2} \shuffle \cdots \shuffle \hat{a} \shuffle W_{a,n} \in \mathbb{Q}[S_n]$ where $W_{a,n}$ is the word $W_{a,n} = \widehat{a+1}\widehat{a+2} \cdots \widehat{n}$, consider the $\mathbb{Q}[S_n^G]$ element \[\hat{B}_a = \sum_{(g_1,g_2,\ldots,g_a)\in G^a}{g_1\hat{1} \shuffle g_2\hat{2} \shuffle g_a\hat{a} \shuffle W_{a,n}};\] as before we abbreviate $e\hat{l}$ as $\hat{l}$, and we carry out the shuffle product as usual, treating each $g_i\hat{i}$ as a single letter. This element describes the action of taking the first $a$ cards of a deck of $n$ cards (by default, the deck is $\hat{1}\hat{2}\cdots\hat{n}$) and then inserting them back into the deck such that a certain face $g_l$ of card $\hat{l}$ is facing up, for each $l \in [a]$. For example, $\hat{1}(g_2\hat{2})\hat{3}\cdots\hat{a}W_{a,n}$ is a term of $\hat{B}_a$ which (as a $G$-permutation) corresponds to the deck obtained by inserting these $a$ cards back into their original positions, with the second card having its $g_2$ face up.

\begin{exm}
In the case $G = \mathbb{Z}/m\mathbb{Z}$, the element $\hat{B}_a$ has an intuitive real-life interpretation. We have a deck of $n$ roulette wheels, each of which has $m$ sectors. We take the first $a$ roulette wheels and insert them back into the deck, while spinning these $a$ wheels at the same time. Each of these $a$ wheels will end up in a certain position in the deck, with a certain sector pointing up.
\end{exm}

\subsection{Top to Random Shuffling Expansion Formula in $\mathbb{Q}[S_n^G]$}
For the rest of this section, assume that $G$ is finite. In the same vein as in the previous section, we will find an expansion formula for the $k$-fold product $\hat{B}_{a_1}\hat{B}_{a_2} \cdots \hat{B}_{a_k}$ in terms of the elements $\hat{B}_c$. $\hat{B}_{a_1}\hat{B}_{a_2} \cdots \hat{B}_{a_k}$ is a sum of terms $\sigma_1\sigma_2 \cdots \sigma_k$, where $\sigma_i \in S_n^G$ is a term of $\hat{B}_{a_i}$. As in the previous sections, we regard two terms $\sigma'_1\sigma'_2 \cdots \sigma'_k$, $\sigma''_1\sigma''_2 \cdots \sigma''_k$ of $\hat{B}_{a_1}\hat{B}_{a_2} \cdots \hat{B}_{a_k}$ as distinct if $(\sigma'_1,\sigma'_2,\ldots,\sigma'_k) \neq (\sigma''_1,\sigma''_2,\ldots,\sigma''_k)$, even if $\sigma'_1\sigma'_2 \cdots \sigma'_k$, $\sigma''_1\sigma''_2 \cdots \sigma''_k$ are equal as products in $S_n^G$.

In analogy to the previous section, each term of $\hat{B}_d$ hits/shuffles the first $d$ cards of the deck. A card is hit by a term ($k$-tuple of $G$-permutations) $\sigma_1\sigma_2 \cdots \sigma_k$ of $\hat{B}_{a_1}\hat{B}_{a_2} \cdots \hat{B}_{a_k}$ if and only if it is hit by some $\sigma_i$ for $i \in [k]$. Thus, each term $\sigma_1\sigma_2 \cdots \sigma_k$ of $\hat{B}_{a_1}\hat{B}_{a_2} \cdots \hat{B}_{a_k}$ hits cards $\hat{1}, \hat{2}, \ldots, \hat{c}$ for a unique $c \in [\max(a_1,\ldots,a_k),\min(\sum_{m=1}^k{a_m},n)]$.

Each term $\sigma_1\sigma_2 \cdots \sigma_k$ of $\hat{B}_{a_1}\hat{B}_{a_2} \cdots \hat{B}_{a_k}$ is equal as a $G$-permutation to a term $t^*$ of $\hat{B}_c$ for some $c \in [\max(a_1,\ldots,a_k),\min(\sum_{m=1}^k{a_m},n)]$. Hence we can partition the set of terms of $\hat{B}_{a_1}\hat{B}_{a_2} \cdots \hat{B}_{a_k}$ into the sets $D_{c,t^*} := \{$terms $\sigma_1\sigma_2 \cdots \sigma_k$ of $\hat{B}_{a_1}\hat{B}_{a_2} \cdots \hat{B}_{a_k}$ that hit the cards $\hat{1}, \hat{2}, \ldots, \hat{c} | \sigma_1\sigma_2 \cdots \sigma_k = t^*$ as $G$-permutations$\}$ for $c \in [\max(a_1,\ldots,a_k),\min(\sum_{m=1}^k{a_m},n)]$ and $t^*$ a term of $\hat{B}_c$. 

It is useful to define the \textbf{absolute value} of a $G$-permutation, to relate $G$-permutations with permutations in $S_n$; this will also allow us to apply some results of the previous section, which dealt with regular permutations. Given a $G$-permutation $\sigma$, we define $\mathrm{abs}(\sigma)$ to be the permutation of $[\hat{n}]$ obtained by erasing the face of every card (in other words, replacing the face of each card by the identity element of $G$). For example, the absolute value of the $G$-permutation $\hat{2}(g_1\hat{1})(g_3\hat{3})\hat{4}\cdots\hat{n}$ (where $g_1, g_3 \in G$) is the regular permutation $\hat{2}\hat{1}\hat{3}\hat{4}\cdots\hat{n}$. For $G$-permutations $\tau_1, \tau_2$ we have $\mathrm{abs}(\tau_1\tau_2) = \mathrm{abs}(\tau_1)\mathrm{abs}(\tau_2)$.

Now fix $c \in [\max(a_1,\ldots,a_k),\min(\sum_{m=1}^k{a_m},n)]$ and $t^*$ a term of $\hat{B}_c$. Notice that for each $\sigma_1\sigma_2 \cdots \sigma_k \in D_{c,t^*}$ we have $\mathrm{abs}(\sigma_1)\mathrm{abs}(\sigma_2) \cdots \mathrm{abs}(\sigma_k) \in D_{c,t}$ (defined in Subsection \ref{shuffles partitions}) where $t = \mathrm{abs}(t^*) \in S_n$. Thus, we can partition $D_{c,t^*}$ into the sets $D_{c,t^*}^s := \{$terms $\sigma_1\sigma_2 \cdots \sigma_k \in D_{c,t^*} | \mathrm{abs}(\sigma_1)\mathrm{abs}(\sigma_2) \cdots \mathrm{abs}(\sigma_k) = s$ as $k$-tuples$\}$ for $s \in D_{c,t}$. Each element of $D_{c,t^{\ast}}$ can be constructed by first choosing $s = \sigma'_1\sigma'_2 \cdots \sigma'_k \in D_{c,t}$, and then choosing the faces for the $a_i$ shuffled cards of each $\sigma'_i$ so that we get all the faces in $t^*$ through the resulting product.

Now we prove a lemma that will be useful in factoring the faces in $t^*$. Let $l$ be a fixed positive integer. For $g,g' \in G$, let $l^{\ast g}$ be the set of $l$-tuples of elements of $G$ such that the product of the $l$ entries equals $g$, and let $l^{\ast g'}$ be the set of $l$-tuples of elements of $G$ such that the product of the $l$ entries equals $g'$.

\begin{lem}
\label{grp elt fac}
We have $|l^{\ast g}| = |l^{\ast g'}|$.
\end{lem}
\begin{proof}
We give a simple bijection $\phi: l^{\ast g} \rightarrow l^{\ast g'}$. Let $(c_1,c_2,\ldots,c_l) \in l^{\ast g}$. We let $\phi$ map $(c_1,c_2,\ldots,c_l) \mapsto ((g'g^{-1})c_1,c_2,\ldots,c_l)$. Since $(g'g^{-1})c_1c_2 \cdots c_l = (g'g^{-1})g = g'$, we have $((g'g^{-1})c_1,c_2,\ldots,c_l) \in l^{\ast g'}$. The inverse map $\phi^{-1}$ is given by $(d_1,d_2,\ldots,d_l) \mapsto ((gg'^{-1})d_1,d_2, \linebreak[1] \ldots,d_l)$ for $(d_1,d_2,\ldots,d_l) \in l^{\ast g'}$.
\end{proof}
% This lemma shows that, for a fixed $k$, every ``side'' of each $G$-card in the deck can be factored in the same number of ways into a $k$-tuple. This will allow me to generalize my signed permutations expansion formula to the algebra of $G$-permutations, with essentially the same proof (along with Lemma \ref{grp elt fac}).

Thus, the number of ways of factoring an element of $G$ into $l$ factors is a constant, which we can take to be $l^{\ast e}$ where $e$ is the identity of $G$. We can use this lemma to calculate $l^{\ast e}$.

\begin{cor}
\label{fac calc}
We have $l^{\ast e} = |G|^{l-1}$.
\end{cor}
\begin{proof}
There are a total of $|G|^l$ $l$-tuples of elements of $G$, each of which factors some element of $G$. By the lemma, these factorization are equidistributed among the $|G|$ elements of $G$. Therefore, we have $l^{\ast e} = |G|^l/|G| = |G|^{l-1}$.
\end{proof}

To construct an element of $D_{c,t^{\ast}}$, fix $s = \sigma'_1 \cdots \sigma'_k \in D_{c,t}$ and fix a card $\hat{m}$ (where $m \in [c]$) of $t^*$ with face $f_m \in G$. By Theorem \ref{genbij}, $D_{c,t}$ is in natural bijection with $Q_c^{a_1,\ldots,a_k}$, and we can identify $s$ with $\phi_c(s) = \alpha = \{\alpha_1,\alpha_2,\ldots,\alpha_c\} \in Q_c^{a_1,\ldots,a_k}$. Starting from the $m$th position in the identity deck, $\hat{m}$ must be hit $|\alpha_m|$ times before moving to its final position in $t^*$ with face $f_m$ up. By (\ref{face mult}), each time a card is hit, it is multiplied (on the right) by a certain face $g \in G$. It follows that $f_m = g_1g_2 \cdots g_{|\alpha_m|}$ for some $(g_1,g_2,\ldots,g_{|\alpha_m|}) \in G^{|\alpha_m|}$, so the face of card $\hat{m}$ in $t^{\ast}$ can be factored in $|\alpha_m|^{\ast e}$ ways by Lemma \ref{grp elt fac}. Since the faces of the cards $\hat{1}, \ldots, \hat{c}$ in $t^*$ can be factored independently of one another, it follows by Corollary \ref{fac calc} that $|D^s_{c,t^{\ast}}| = \prod_{i=1}^c{|\alpha_i|^{\ast e}} = \prod_{i=1}^c{|G|^{|\alpha_i|-1}} = |G|^{(\sum_{i=1}^k{a_i})-c}$ irrespective of $s$.

Consequently, we have \[|D_{c,t^{\ast}}| = \sum_{s \in D_{c,t}}{|D^s_{c,t^{\ast}}|} = \sum_{s \in D_{c,t}}{|G|^{(\sum_{i=1}^k{a_i})-c}} = |Q_c^{a_1,\ldots,a_k}||G|^{(\sum_{i=1}^k{a_i})-c}\] irrespective of $t^{\ast}$. Since the sets $|D_{c,t^{\ast}}|$ (for $t^*$ a term of $\hat{B}_c$) are all equinumerous with cardinality $|Q_c^{a_1,\ldots,a_k}||G|^{(\sum_{i=1}^k{a_i})-c}$, we therefore obtain the expansion formula \begin{equation}\label{expform} \hat{B}_{a_1}\hat{B}_{a_2} \cdots \hat{B}_{a_k} = \sum_{c=\max(a_1,\ldots,a_k)}^{\min(\sum_{m=1}^k{a_m},n)}{|Q_c^{a_1,\ldots,a_k}||G|^{(\sum_{i=1}^k{a_i})-c}\hat{B}_c}, \end{equation} where $|Q_c^{a_1,\ldots,a_k}|$ is given in (\ref{cardform}).

Now we briefly look at how (\ref{expform}) can be used in computations.
\begin{exm}
\label{g-perm calc}
Suppose we have a deck of $n$ cards $\hat{1}\cdots\hat{n}$ in which each card has $|G|$ faces each of which is labeled by an element of $G$. How many ways can the deck $\hat{i}(g_{i-1}\widehat{i-1})\cdots(g_1\hat{1})\widehat{i+1}\cdots\hat{n}$ (where $g_l \in G$ for $l \in [i-1]$) be obtained via top-to-random shuffling (in which the shuffled cards are flipped randomly and independently) of $a_1$ cards, $a_2$ cards, $\ldots$, $a_k$ cards in that order? We answer this question by looking at the product $\hat{B}_{a_1}\hat{B}_{a_2} \cdots \hat{B}_{a_k}$ and counting all the copies of $\hat{i}(g_{i-1}\widehat{i-1})\cdots(g_1\hat{1})\widehat{i+1}\cdots\hat{n}$ appearing on the right hand side of (\ref{expform}). By Fact \ref{perm as term}, $\hat{i}(g_{i-1}\widehat{i-1})\cdots(g_1\hat{1})\widehat{i+1}\cdots\hat{n}$ is a term of $\hat{B}_c$ for any $c \geq i-1$, and each such $\hat{B}_c$ contains exactly one copy of $\hat{i}(g_{i-1}\widehat{i-1})\cdots(g_1\hat{1})\widehat{i+1}\cdots\hat{n}$. It follows that there are \[\sum_{c=\max(\max(a_1,\ldots,a_k),i-1)}^{\min(\sum_{m=1}^k{a_m},n)}{|Q_c^{a_1,\ldots,a_k}||G|^{(\sum_{i=1}^k{a_i})-c}}\] ways to obtain this deck.

To compute the probability of obtaining $\hat{i}(g_{i-1}\widehat{i-1})\cdots(g_1\hat{1})\widehat{i+1}\cdots\hat{n}$ via this sequence of $k$ top to random shuffles, notice that $\hat{B}_{a_1}\hat{B}_{a_2} \cdots \hat{B}_{a_k}$ has in total $(|G|^{a_1}P(n,a_1))(|G|^{a_2}P(n,a_2))\cdots \linebreak[1] (|G|^{a_k}P(n,a_k)) = |G|^{\sum_{m=1}^{k}{a_m}}\prod_{m=1}^{k}{P(n,a_m)}$ terms/decks. Thus, this probability is \[\frac{\sum_{c=\max(\max(a_1,\ldots,a_k),i-1)}^{\min(\sum_{m=1}^k{a_m},n)}{|Q_c^{a_1,\ldots,a_k}||G|^{(\sum_{i=1}^k{a_i})-c}}}{|G|^{\sum_{m=1}^{k}{a_m}}\prod_{m=1}^{k}{P(n,a_m)}}.\]
\end{exm}

\subsection{Other Generalizations}
Lemma \ref{grp elt fac} can be used to generalize other expansion formulae of $\mathbb{Q}[S_n]$ to $\mathbb{Q}[S_n^G]$. We show below one way this can be done.

For each $\sigma = \hat{c_1}\hat{c_2} \cdots \hat{c_n} \in S_n$ define $\bar{\sigma} \in \mathbb{Q}[S_n^G]$ by $\bar{\sigma} = \sum_{(g_1,g_2,\ldots,g_n) \in G^n}{(g_1\hat{c_1})(g_2\hat{c_2})\cdots(g_n\hat{c_n})}$; in other words, $\bar{\sigma}$ is the sum of all $G$-permutations whose absolute value is $\sigma$. For example, if $\sigma = \hat{2}\hat{1}\hat{4}\hat{3}\hat{5} \cdots \hat{n}$, then $\bar{\sigma}$ is the sum $\sum_{(g_1,g_2,\ldots,g_n) \in G^n}{(g_1\hat{2})(g_2\hat{1})(g_3\hat{4})(g_4\hat{3})(g_5\hat{5})\cdots(g_n\hat{n})}$.

Let ${\cal C} = \{B_r\}_{r \in I}$ be a collection of linearly independent elements of $\mathbb{Q}[S_n]$, where $I$ is an indexing set. Suppose $B_{p_1}, B_{p_2}, \ldots, B_{p_k} \in {\cal C}$ are positive linear combinations of elements of $S_n$ such that \begin{equation}\label{arb exp form}
B_{p_1}B_{p_2} \cdots B_{p_k} = \sum_{r \in I}{C_r^{p_1,\ldots,p_k}B_r},
\end{equation} where the coefficients $C_r^{p_1,\ldots,p_k}$ are all nonnegative.

Since both sides of (\ref{arb exp form}) consist of positive terms (they are both positive sums of elements of $S_n$), both sides must have exactly the same terms with the same multiplicity; no cancellation occurs on either side. Then the terms of the product $B_{p_1}B_{p_2} \cdots B_{p_k}$ can be partitioned into the sets $D'_r$ ($r \in I$) which correspond to the terms of $C_r^{p_1,\ldots,p_k}B_r$. Each term ($k$-tuple of permutations) $\sigma_1\sigma_2 \cdots \sigma_k \in D'_r$ must equal $t$ as permutations, for some term (permutation) $t$ of $B_r$. (\ref{arb exp form}) tells us that we can partition each $D'_r$ into sets $D'_{r,t} = \{$terms $\sigma_1 \cdots \sigma_k$ of $D'_r | \sigma_1 \cdots \sigma_k = t$ as permutations$\}$ (for $t$ a term of $B_r$) with $|D'_{r,t}| = C_r^{p_1,\ldots,p_k}$.

Define $\bar{B}_r := \sum{\bar{\sigma}}$ where the sum is over all terms $\sigma$ of $B_r$. Notice here that every term ($G$-permutation) of $\bar{B}_r$ is a deck in which \textit{each} of the $n$ cards has been assigned a face, in contrast with a term of $\hat{B}_c$ (from the previous subsection) which is a deck in which only the cards $\hat{1}, \ldots, \hat{c}$ have been assigned a face. We will generalize (\ref{arb exp form}) to $\mathbb{Q}[S_n^G]$ by deriving an expansion formula for $\bar{B}_{p_1}\bar{B}_{p_2} \cdots \bar{B}_{p_k}$.

By (\ref{arb exp form}), for each term ($k$-tuple of $G$-permutations) $\sigma'_1\sigma'_2 \cdots \sigma'_k$ of $\bar{B}_{p_1}\bar{B}_{p_2} \cdots \bar{B}_{p_k}$ we have $\mathrm{abs}(\sigma'_1)\mathrm{abs}(\sigma'_2) \cdots \mathrm{abs}(\sigma'_k) \in D'_{r,t}$ for some $r \in I$ and $t$ a term of $B_r$, so $\sigma'_1\sigma'_2 \cdots \sigma'_k = t^* \in S_n^G$ such that $t^*$ is a term of $\bar{B}_r$ and $\mathrm{abs}(t^*) = t$. We can thus partition the terms of $\bar{B}_{p_1}\bar{B}_{p_2} \cdots \bar{B}_{p_k}$ into the sets $D'_{r,t^{\ast}} := \{$terms $\sigma'_1\sigma'_2\cdots\sigma'_k$ of $\bar{B}_{p_1}\bar{B}_{p_2} \cdots \bar{B}_{p_k} | \sigma'_1\sigma'_2\cdots\sigma'_k = t^{\ast}$ as $G$-permutations$\}$ for $r$ ranging through $I$ and $t^*$ a term of $\bar{B}_r$. We can further partition $D'_{r,t^*}$ into the sets $D'_{r,t^*,s} := \{$terms $\sigma'_1\sigma'_2 \cdots \sigma'_k \in D'_{r,t^*} | \mathrm{abs}(\sigma'_1)\mathrm{abs}(\sigma'_2) \cdots \mathrm{abs}(\sigma'_k) = s$ as $k$-tuples$\}$ for $s \in D'_{r,\mathrm{abs}(t^*)}$. Each element of $D'_{r,t^{\ast}}$ can be constructed by first choosing $s = \sigma_1\sigma_2 \cdots \sigma_k \in D'_{r,\mathrm{abs}(t^*)}$, and then choosing the faces for the $n$ cards of each $\sigma_i$ so that we get all the faces in $t^*$ through the resulting product.

% Does the notion of "hit" apply here?
To construct an element of $D'_{r,t^{\ast}}$, fix $s = \sigma_1 \cdots \sigma_k \in D'_{r,\mathrm{abs}(t^*)}$ and fix a card $\hat{m}$ of $t^*$ with face $f_m \in G$. Starting from the $m$th position in the identity deck, $\hat{m}$ must be acted upon $k$ times before moving to its final position in $t^*$ with face $f_m$ up. By (\ref{face mult}), each time a card is acted upon, it is multiplied (on the right) by a certain face $g \in G$. By Corollary \ref{fac calc}, $f_m$ can be factored in $k^{*e} = |G|^{k-1}$ ways. Since this is true for each card of $t^*$, there are $|D'_{r,t^*,s}| = (|G|^{k-1})^n$ ways of choosing the faces for all $n$ cards, irrespective of $s$. Since there are $|D'_{r,\mathrm{abs}(t^*)}| = C_r^{p_1,\ldots,p_k}$ ways of choosing $s$, we have $|D'_{r,t^{\ast}}| = C_r^{p_1,\ldots,p_k}(|G|^{k-1})^n$ irrespective of $t^*$.

Therefore, (\ref{arb exp form}) generalizes to the expansion formula \begin{equation}
\bar{B}_{p_1}\bar{B}_{p_2} \cdots \bar{B}_{p_k} = \sum_{r \in I}{C_r^{p_1,\ldots,p_k}(|G|^{k-1})^n\bar{B}_r}
\end{equation}
in $\mathbb{Q}[S_n^G]$.

% Section 5
\section{Further Discussions}
\label{further}
More generally, Garsia \cite{garsiatop} defined the elements $B_p$ of
$\mathbb{Q}[S_n]$, where $p =
(p_1, \ldots, p_k)$ is a composition of $n$, as follows. Define the segmentation
$E(p) = (E_1,E_2,\ldots,E_k)$ of the deck $12 \ldots n$ into successive
factors $E_i =
(p_1+p_2+\ldots+p_{i-1}+1)(p_1+p_2+\ldots+p_{i-1}+2)\ldots(p_1+p_2+\ldots p_i)$
for $i = 1, 2, \ldots, k$. Then $B_p$ is defined as the $\mathbb{Q}[S_n]$
element $B_p = E_1
\shuffle E_2 \shuffle \ldots \shuffle E_k$. The $B_a$ we studied earlier
corresponds to the special case that $p = (1^a,n-a)$. The $B_p$ form a basis of
Solomon's Descent Algebra, which is a subalgebra of $\mathbb{Q}[S_n]$ studied in
depth by Garsia and Reutenauer \cite{garsiades}. It would be interesting to find
an expansion formula for $B_p^k$.

Furthermore, there may be other similar algebraic objects for whom expansion
formulae may be found using similar bijections. For example, promotion operators were defined by Stanley \cite{stanley} and Sch{\"u}tzenberger \cite{schutz}, and then generalized by Ayyer, Klee, and Schilling \cite{aks}, who used these extended promotion operators in their recent work on Markov chains. Let $P$ be an arbitrary poset
of size $n$, with
partial order $\preceq$. We assume that the vertices of $P$ are naturally
labeled by elements in $[n]$. Let ${\cal L} := {\cal L}(P)$ be the set of its
linear extensions, ${\cal L}(P) = \{\pi \in S_n | i \prec j$ in $P$
$\Rightarrow$ $ \pi_i^{-1} < \pi_j^{-1}$ as integers$\}$, which is naturally
interpreted as a subset of $S_n$.

The extended promotion operator can be expressed in terms of more elementary
operators $\tau_i$ ($1 \leq i < n$). Let $\pi = \pi_1 \ldots \pi_n \in {\cal
L}(P)$ be in one-line notation. Place the label $\pi^{-1}_i$ in $P$ at the
location $i$. Then $\tau_i$ acts on $\pi$ on the left by
\begin{enumerate}
\item interchanging $\pi_i$ and $\pi_{i+1}$ if they are not comparable in $P$
\item fixing $\pi$ otherwise.
\end{enumerate}
Then as an operator on ${\cal L}(P)$ we have $\partial_j = \tau_j \tau_{j+1}
\cdots \tau_{n-1}$. Interpreted in terms of card shuffling, in the case that $P$
is an antichain, $\partial_j$ moves the $j$th card to the end of the deck.
Consequently, in this case we have $B_1 = \sum_{j=1}^n{\partial_j^{-1}}$. It
would be interesting to find expansion formulae for
$(\sum_{j=1}^n{\partial_j^{-1}})^k$ in the case that $P$ is not necessarily the
antichain, and calculate the transition matrix and eigenvalues of $(\sum_{j=1}^n{\partial_j^{-1}})^k$.

Finally, we point out that the converse to Fact \ref{perm injec beg} is also true; we give a way by which any permutation $\sigma$ uniquely determines the $B_c$ of which $\sigma$ is a term. By looking at the positions to which the cards $1, 2, \ldots, m_{\sigma}-1$ are sent, we can treat $\sigma$ as an injection. 
\begin{prop}
There is a one-to-one correspondence $\chi$ between $S_n$ and the set $S^{inj} := \{$injective maps $f$ from $[a]$ to $[n]$ such that $f(i) > a$ for some $i < f(a)$, for $a = 0, 1, ..., n-1\}$.
\end{prop}
\begin{proof}
Using Fact \ref{perm as term}, we simply set $\sigma$ to be a term of $B_{m_{\sigma}-1}$. By looking at the positions in the deck to which the cards $1, 2, \ldots, m_{\sigma}-1$ are sent, we determine the injection $\chi(\sigma) : [m_{\sigma}-1] \rightarrow [n]$.
\end{proof}

\begin{exm}
We consider $\sigma = 43215 \cdots n$ to be a term of $B_3$. $\sigma$ corresponds to the injection $\chi(\sigma) : [3] \rightarrow [n]$ given by $\chi(\sigma)(1) = 4, \chi(\sigma)(2) = 3, \chi(\sigma)(3) = 2$.
\end{exm}

This gives us the option of studying permutations by viewing them as injections, and gives us a group structure (inherited from $S_n$) on the set $S^{inj}$ of injections.

\end{document}